\documentclass[12pt]{amsart}
\usepackage{amsmath}
\usepackage{amssymb}
\usepackage{amsfonts}
\usepackage{textcomp}
\usepackage{amsthm}
\usepackage{mathrsfs}
\usepackage[latin1]{inputenc}
\usepackage[all]{xy}
\usepackage{hyperref}
\usepackage{url}
\usepackage{graphicx}
\usepackage{tikz-cd}
\usepackage[paperheight=11in,paperwidth=8.5in,left=1in,top=1.25in,right=1in,bottom=1.25in,]{geometry}

\newtheorem{theorem}{Theorem}
\newtheorem{lem}{Lemma}
\newtheorem{prop}{Proposition}
\newtheorem{defi}{Definition}

\begin{document}
\title[Continuity of the renormalized volume under geometrically finite limits]{\textbf{Continuity of the renormalized volume under geometrically finite limits}}
\author{Franco Vargas Pallete}
\thanks{Research partially supported by NSF grant DMS-1406301}
\address{Department of Mathematics  \\
 University of California at Berkeley \\
775 Evans Hall \\
Berkeley, CA 94720-3860 \\
U.S.A.}
\email{franco@math.berkeley.edu}

\begin{abstract}
We extend the concept of renormalized volume for geometrically finite hyperbolic $3$-manifolds, and show that is continuous for  geometrically convergent sequences of hyperbolic structures over an acylindrical 3-manifold $M$ with geometrically finite limit. This allows us to show that the renormalized volume attains its minimum (in terms of the conformal class at $\partial M = S$) at the geodesic class, the conformal class for which the boundary of the convex core is totally geodesic.
\end{abstract}
\maketitle

\section{Introduction}

Renormalized volume is a quantity that gives a notion of volume for hyperbolic manifolds which have infinite volume under the classical definition. Its study for convex co-compact hyperbolic 3-manifolds can be found in \cite{KrasnovSchlenker}, while the geometrically finite case which includes rank-1 cusps has been developed in \cite{MoroianuGuillarmouRochon}. In this work we extend the concept to geometrically finite hyperbolic $3$-manifolds.

The article is organized as follows. Section \ref{sec:sequences} will give the necessary understanding of geometrically finite sequences with geometrically finite limit for their use in the main results. More precisely, we will see that sufficiently short geodesics are unlinked and parallel with respect to the boundary (Proposition \ref{prop:1}). Section \ref{sec:renormalizedvolume} will define renormalized volume for geometrically finite hyperbolic $3$-manifolds, along with some of its properties. In there we will do the necessary work to extend the methods of \cite{MoroianuGuillarmouRochon} to show that the renormalized volume ($V_\text{R}$) is continuous under geometric limit (Theorem \ref{theorem:main}), as well as it being comparable to the volume of the convex core (Theorem \ref{theorem:comparison}). Finally, section \ref{sec:consequences} uses the continuity under geometric limits to show that in an acylindrical manifold the geodesic class is the global minimum. In order to do that, we assume that a sequence going to the infimum exits the deformation space to a metric $g_0$. By Section \ref{sec:sequences} we know that the cusps appearing at the limit are unlinked and parallel to the boundary. Using \cite{BBCL} description of converging sequences of Kleinian groups, we create new sequences that converge to metrics nearby $g_0$. Then, by continuity (Theorem \ref{theorem:main}) $g_0$ will need to be a critical point of $V_\text{R}$, which is impossible by a volume comparison.

\textbf{Acknowledgements:} I would like to thank Ian Agol for his guidance during this project.  I would also like to thank Ian Biringer, Richard Canary and Jessica Purcell for their helpful comments.

\section{On geometrically finite sequences with geometrically finite limit}\label{sec:sequences}

Let $\{G_n\}$ be a sequence of Kleinian groups corresponding to geometrically finite hyperbolic structures over $M$ converging geometrically to $H$, which is also a geometrically finite group. Since $M$ is acylindrical, we can further assume that the sequence converges algebraically after taking $\{G_n\}$ to be the image of a representation $\rho_n:\Pi \rightarrow PSL(2,\mathbb{C})$ for a fix $\Pi$ (\cite{ThurstonAcy}, Theorem 1.2). Let $\rho_0:\Pi \rightarrow PSL(2,\mathbb{C})$ be the limit, with image $G$. Naturally $G<H$, hence $\mathbb{H}^3/G=:M_0$ (which is homeomorphic to $M$) is a covering space of $\mathbb{H}^3/H =: N$.

Thanks to \cite{JorgensenMarden} (section 4.5) we know that there exists a fundamental polyhedron $P \subseteq \mathbb{H}^3$ for $H$ such that when we pull back its pairing transformations to $\{G_n\}$ (n sufficiently large), we obtain a fundamental polyhedron $P_n$ for $\{G_n\}$. Take $\epsilon > 0$ a sufficiently small Margulis constant such that the $\epsilon$-thin part of $N$, $N^{<\epsilon}$, is a union of finitely many cusps (recall that $N$ is geometrically finite). This fundamental polyhedron $P$ can be taken generically, meaning that its center can be chosen among a dense sets in $\mathbb{H}^3$, which can be found in \cite{JorgensenMardenPoly} as the main result. We recall how this generic fundamental polyhedron behaves when considering cusps.

A rank-1 cusp appears in $P$ as the pairing of two tangent faces, being the thin part a neighbourhood of the point of tangency. In the conformal boundary appears as two cusps, either from different ends of $N$ or as two different cusps from the same end. Hence for sufficiently large balls $B$, the components of $\partial B \bigcap P$ that are not tori (which will correspond to rank-$2$ cusps) are the ends of $N$ joined by their paired cusps. Then when we look at $P_n$ we have surfaces (after identifications) of the same topological type at their intersection with $\partial B$, which will face an end of $M$. Note finally that the parabolic subgroup of each cusp are limits of peripheral elements.

A rank-2 cusp appears as the pairing of 6 faces (or 4) with a common point. For sufficiently large balls $B$, the component of $\partial B \bigcap P$ intersecting these faces identifies to a torus which is isotopic to an horo-torus. Hence the embedding of this torus has image in $\pi_1$ the parabolic subgroup associated to the $2$-cusp. Then when we pull-back to $\{G_n\}$ we have an embedded torus with image in $\pi_1$ an abelian subgroup $\langle \gamma\rangle$ that degenerates to the parabolic subgroup. Hence this torus fills a solid torus with a geodesic in its interior that correspond to $\gamma$. By process of elimination this solid torus is the region of $P_n$ exterior to $B$ corresponding to the 6 (or 4) faces of $P$, after we have made the identifications.

By the last two paragraphs, we can conclude the topological type of $N$ is $M$ with some curves drilled out. These curves correspond to short geodesics in $M_n = \mathbb{H}^3/G_n$ that end limiting to rank-$2$ cusps in $N$.

In order to throw some light on which curves could we end up drilling-out, let us introduce the concept of unknottedness and unlinkedness of a curve or curves with respect to a surface.

Let $M$ be a compact 3-manifold with a compact embedded surface $S$ and simple closed curve $C$. We say that $\gamma$ is \textit{unknotted with respect to} $S$ if $\gamma$ lies in a surface $S'$ isotopic to $S$. Analogously, we say that a collection of simple closed curves $\{ \gamma_1,\ldots, \gamma_n\}$ is \textit{unlinked with respect to} $S$ if  there is a collection of parallel disjoint surfaces $\{ S_1,\ldots, S_n\}$ isotopic to $S$ such that $\gamma_i\in S_i$ for all $i$.

The problem of short geodesics on a hyperbolic $M$ being unknotted/unlinked has been studied by \cite{Otal03} for bundles over the circle (w.r.t. a fiber); \cite{Otal95}, \cite{Otal03}, \cite{Souto} for compression bodies (w.r.t. its boundary); and \cite{Souto}, \cite{Breslin} for closed manifolds (w.r.t. a Heegaard surface). Here we are going to establish unlinkedness w.r.t. the boundary for $M$ acylindrical. The proof is a slight modification of \cite{Souto} for the quasifuchsian case.

\begin{prop} \label{prop:1} Let $M$ be compact acylindrical 3-manifold. Then for any $K>0$ there exists a constant $\epsilon >0$ (depending only on $K$ and the topological type of $M$) such that for any hyperbolic structure on the interior of $M$ with volume of the convex core at most $K$ we have: any finite collection of closed geodesics each with length less than $\epsilon$ is unlinked w.r.t. $\partial M$.
\end{prop}

\begin{proof}
Assume that the result is false. Then we have a sequence of geometrically finite hyperbolic structures over $M$ with some knotted geodesics getting arbitrarily small. We can assume that this sequence converges algebraically and geometrically, and since the volume of the convex core is uniformly bounded above by $K$, we know that the limit is geometrically finite as well. Hence we have the previous description for this kind of limit, so let us borrow the notation. In particular, short geodesics (when they are shorter than the Margulis constant that we considered before) in $M_n$ for n large correspond to the cusps of the geometric limit $N$. Denote those curves by $\Gamma_n$. Since the description for rank-$1$ cusps already makes them boundary parallel, for the rest of the proof we are going to work as if only we had rank-$2$ cusps appearing on the limit.

Consider the covering map $p: M_0 \rightarrow N$. Take a compact core $C_0$ of $M_0$ (look at \cite{Scott73} for their existence). Since it has compact image by $p$, we can pull it back to $M_n$ for n large as a map $p_n: C_0 \rightarrow M_n$. At the level of fundamental groups, this process can be express by the following commutative diagram:

\begin{center}
\begin{tikzcd}
G \arrow[hook]{r}
\arrow{rd}[swap]{\cong}
&H \arrow{d}\\
&G_n
\end{tikzcd}
\end{center}

The inclusion corresponds to $p$, the realization of $N$ as $M_n\setminus\Gamma_n $ induces the vertical map, and their composition is the isomorphism given by algebraic convergence. Then the map $p_n:C_0\rightarrow M_n$ induces an isomorphism of the fundamental groups, and since these spaces are aspherical, it is also a homotopic equivalence.

Observe that the short geodesics are outside the image of $p_n$, and the homotopy equivalence defines a homotopy between  $\partial C_0$ (under $p_n$) and $\partial M_n$ (here we mean the manifold together with its conformal boundary). This allows us to conclude that there is no homotopy between $\partial C_0$ (under $p_n$) and $\partial M_n$ that doesn't intersect $\Gamma_n$. If such a homotopy existed, we would have a homotopy between $id_{M_n}$ and a map from $M_n$ landing inside $p_n(C_0)$, such that  $\partial M_n$ never intersects $\Gamma_n$. Then the degree of any point of $\Gamma$ is well-defined but different at the ends of the homotopy, which is a contradiction. 

Now we proceed as \cite{Souto}. We can define a metric $g$ in $M_n\setminus\Gamma_n$ with the following properties thanks to (a proof of this can be found on \cite{Agol02})

\begin{lem}\label{lem:1}
For every $\epsilon_0$ positive there is $\epsilon > 0$ such that the following holds: If $\Gamma$ is a collection of geodesics in a hyperbolic manifold $M$ which are shorter than $\epsilon$, then there is a complete Riemannian metric $g$on $M\setminus \Gamma$ with curvature pinched in $[-2;-1/2]$ and which coincides with the original hyperbolic metric outside of components of the $\epsilon_0$-thin part containing components of $\Gamma$.
\end{lem}

Foliate a triangle $\Delta\subseteq\mathbb{R}^2$ by segments with one endpoint a fixed vertex and the other endpoint at the opposite edge. An immersion $f:\Delta \rightarrow M$ is a \textit{ruled triangle} if every edge and leaf of the foliation is mapped to a geodesic segment. A continuous map $\phi$ from a surface $S$ to $M$ is a \textit{simplicial ruled surface} if there is a triangulation of $S$ such that $\phi$ is a ruled triangle at each piece of the triangulation and the cone angle at each vertex is at least $2\pi$. If, as in our case, $M$ has curvature pinched in $[-2;-1/2]$, then the pullback metric by $\phi$ to the universal cover of $S$ is complete and $CAT(-1/2)$, and in particular (as in \cite{Souto}) $\text{vol}(S) \leq 4\pi\vert\chi(S)\vert$. If we mandate one closed edge $I$ to be on the triangulation and to be mapped to a geodesic $\eta_0$ in $M$, we say that $\phi$ \textit{realizes} $\eta_0$.

\begin{lem}[\cite{Bonahon86} (\cite{Canary93}, Section 2)] \label{lem:2}
Let $S$ be a closed surface with $\eta\subset S$ an essential simple closed curve and $M$ a complete Riemannian $3$-manifold with curvature pinched in $[-2,-1/2]$. If $f:S\rightarrow M$ is a $\pi_1$-injective map such that $f(\eta)$ is homotopic to a geodesic $\eta_0$, then there is a simplicial ruled surface $\phi:S\rightarrow M$ homotopic to $f$ which realizes $\eta_0$.
\end{lem}

\begin{lem}\label{lem:3}
For all $A>0$ there are positive constants $L,\epsilon_L$ such that for every complete hyperbolic $3$-manifold $M$ with curvature pinched in $[-2,-1/2]$  and every $\pi_1$-injective simplicial ruled surface $\phi:S\rightarrow M$ with $-\chi(S)\leq A$ we have:
\begin{itemize}
	\item For every point $x\in S$ there is a non-homotopically trivial loop $\gamma_x$ in $S$ based at $x$ which is shorter than $L$. This loop can be chosen to be simple.
	\item If there are two loops of length $\leq L$ in $M$ based at the same point and at least one of them has length $\leq\epsilon_L$, then they generate an abelian subgroup of $\pi_1(M)$.
\end{itemize}
\end{lem}

Claim: Let $\mathcal{N}(\Gamma_n)$ be the normal neighbourhood of $\Gamma_n$ given by the $\epsilon_L$-thin part. Then there is a closed curve $\alpha$ in $\partial \mathcal{N}(\Gamma_n)$ homotopic in $M_n\setminus \mathcal{N}(\Gamma_n)$ to a essential simple closed curve $\beta$ in $\partial M_n$

\begin{proof}
To prove the claim proceed as follows. Let $L$ and $\epsilon_L$ be the constants provided by Lemma \ref{lem:3} for $A=\vert\chi(\partial M)\vert$, and $\epsilon$ the constant provided by Lemma \ref{lem:1} for $\epsilon_0 = \epsilon_L$. Furthermore, assume that all closed geodesics considered in $\Gamma_n$ are shorter than $\epsilon$. There is an essential simple curve $\eta$ in $\partial M_n$ such that the homotopy with $p_n(C_0)$ gives a mapped cylinder from this curve to $p_n(\partial C_0)$ that intersects $\Gamma_n$ non-trivially (follows from the nonexistence of a homotopy that evades $\Gamma_n$). Since $\eta$ is essential and $\partial M_n$ incompressible, both ends of the mapped cylinder are homotopic to a closed geodesic in $M_n$. Then at least one end is not homotopic to this closed geodesic in $M_n\setminus\mathcal{N}(\Gamma_n)$ (could be both if for example this geodesic lies in $\mathcal{N}(\Gamma_n)$). Let us examine each case:

\begin{enumerate}
	\item $\eta$ is not homotopic to a geodesic in $M_n\setminus\mathcal{N}(\Gamma_n)$:
	
	Let $g$ the metric in $M_n\setminus\Gamma_n$ given by Lemma \ref{lem:1}. Since $\eta$ is non-trivial in $\pi_1(M_n\setminus\Gamma_n)$, it could either represent a parabolic element or be homotopic to a closed geodesic $\eta_0$. If it's parabolic then it is homotopic in $M_n\setminus\mathcal{N}(\Gamma_n)$ to a curve in $\partial\mathcal{N}(\Gamma_n)$, so we only need to see what happens in the latter case.
	
	Now, $\eta_0$ has a point $x$ in $\mathcal{N}(\Gamma_n)$, otherwise it will be also a geodesic of $M_n$ and will contradict this case assumption. By Lemma \ref{lem:2} there is a simplicial ruled surface $\phi$ homotopic to $\partial M_n$ realizing $\eta_0$. By Lemma \ref{lem:3} there is a simple essential curve $\gamma_x$ in $\phi$ with length less than $L$. But there is a longitude in $\mathcal{N}(\Gamma_n)$ based at $x$ with length less than $\epsilon_L$, and by Lemma \ref{lem:3} commutes with $\gamma_x$. Hence $\gamma_x$ is on the parabolic subgroup corresponding to a cusps. That makes the essential simple closed curve in $M_n$ corresponding to $\gamma_x$ homotopic in $M_n\setminus\mathcal{N}(\Gamma_n)$ to a curve lying on $\partial\mathcal{N}(\Gamma_n)$.
	
	\item $p_n(\eta)$ is not homotopic to a geodesic in $M_n\setminus\mathcal{N}(\Gamma_n)$:
	
	Let again $g$ be the metric in $M_n\setminus\Gamma_n$ given by Lemma \ref{lem:1}. Again $p_n(\eta)$ is parabolic or homotopic to a closed geodesic. If parabolic, then the closed geodesic of $M_n$ homotopic to $p_n(\eta)$ is a multiple of a component of $\Gamma_n$. Then neither $\eta$ is homotopic to a geodesic (of $M_n$) in $M_n\setminus\mathcal{N}(\Gamma_n)$. Swap then to case 1.
	
	By an analogous method as the latter part of the previous case, we obtain an essential closed curve of $\partial C_0$ such that its image by $p_n$ can be homotoped to $\partial\mathcal{N}(\Gamma_n)$, hence to one component of $\Gamma_n$. Then that closed curve in $\partial M_n$ is not homotopic to a geodesic (of $M_n$ in $M_n\setminus\mathcal{N}(\Gamma_n)$. Apply case 1 to this new essential simple closed curve.
\end{enumerate}
\end{proof}

Observe that $\alpha$ can't be a multiple of a meridian since this will imply that $\beta$ is trivial, and this is impossible since $\partial M_n$ is incompressible. Also $\alpha$ can't be properly divided by the core of $\mathcal{N}(\Gamma_n)$. To see this, observe that $M_n$ is acylindrical, so there is a hyperbolic structure with convex core totally geodesic. At its boundary there is a simple closed geodesic homotopic to $\beta$, so this element can't be divided in $M_n$, so neither does $\alpha$. In conclusion, $\alpha$ needs to be a longitude. Name $\gamma$ the component of $\Gamma_n$ that corresponds to $\alpha$

Now, thanks to the cylinder theorem (as in \cite{Marden}, Section 3.7, pp. 129), there is an embedded cylinder in $M_n\setminus\mathcal{N}(\Gamma_n)$ between $\alpha$ and $\beta$. And because $\alpha$ is a longitude, there is an embedded cylinder in $M_n\setminus(\Gamma_n\setminus\gamma)$ with ends $\beta$ and $\gamma$. We can use this cylinder to isotope $\gamma$ to $\partial M_n$ within $M_n\setminus(\Gamma_n\setminus\gamma)$. To finish the proof, observe that the procedure to find an embedded cylinder is valid for any subcollection of curves of $\Gamma_n$, so we can isotope them to parallel surfaces one at the time. Induction on the number of components of $\Gamma_n$ closes the argument.

\end{proof}

Notice that we could have required instead $M$ to be incompressible and with boundary indivisible (i.e. every essential simple closed curve of the boundary can't be non-trivially divided by an element of $\pi_1(M)$).

\section{Renormalized volume for geometrically finite manifolds}\label{sec:renormalizedvolume}



To prove our main theorem, we are going to adapt Theorem 5 of \cite{MoroianuGuillarmouRochon}:

\begin{theorem}[Theorem 5, \cite{MoroianuGuillarmouRochon}]\label{theorem:main}
Let $g_\epsilon$ be an admissible degeneration of convex co-compact hyperbolic metrics on $M$ with geometric limit a geometrically finite hyperbolic metric $g_0$ over $N$. Then:
\begin{equation}
	\lim_{\epsilon\rightarrow 0} \textnormal{Vol}_R(M, g_\epsilon) = \textnormal{Vol}_R(N,g_0).
\end{equation}
\end{theorem}

A sequence of metrics is admissible (as in \cite{MoroianuGuillarmouRochon}) if:
\begin{enumerate}
	\item There are finitely many peripheral curves pinching to rank-$1$ cusps, and for the group elements $g_j$ associated to those curves the multiplier $\exp^{l_j+  2\pi i\alpha_j }$ ($l_j>0$) satisfies $\lim l_j = \lim \alpha_j = 0, \lim \alpha_j/l_j \in \mathbb{R}$. \label{multipliercondition}
	\item The metrics associated to $g_\epsilon$ converge smoothly to the metric of $g_0$ on compact subsets of $\overline{M}$ outside a neighborhood of the cusps, when extended conformally to the boundary.
\end{enumerate}

Let us show that a sequence of convex co-compact metrics with geometrically finite limit satisfies (\ref{multipliercondition}) by contradiction. Since $g_j$ is converging to a parabolic element, its multiplier converges to $1$. Then (after taking a subsequence) $\lim l_j = \lim \alpha_j = 0, \lim \vert \alpha_j/l_j\vert = +\infty $. Let $k_j\in\mathbb{Z}$ be a sequence to be determined and denote by $\lfloor k_j/\alpha_j\rfloor$ the greatest integer of $k_j/\alpha_j$ and look at $g_j^{\lfloor k_j/\alpha_j\rfloor}$. Its multiplier $\exp^{ l_j \lfloor \frac{k_j}{\alpha_j}\rfloor + 2\pi i \alpha_j \lfloor \frac{k_j}{\alpha_j}\rfloor }$ converges to $1$ as long as $\lim l_j \lfloor \frac{k_j}{\alpha_j}\rfloor = 0$. Under this condition $g_j^{\lfloor k_j/\alpha_j\rfloor}$ converges to a parabolic element. This parabolic element has the same fixed point as the limit of $g_j$. For simplicity, let this common fixed point be $0$, which is the limit of $p_+,p_-$; the attractive and repulsive fixed points of $g_j$. With this, the limit of $g_j$ is $\left[  \begin{array}{ c c }     1 & 0 \\   a & 1 \end{array} \right]$, where $a = \lim\frac{1-\exp^{l_j+  2 \pi i\alpha_j}}{p_- - p_+}$. Hence if $\lim\frac{1-\exp^{l_j+  2 \pi i\alpha_j)}}{1-\exp^{ l_j \lfloor \frac{k_j}{\alpha_j}\rfloor + 2\pi i \alpha_j \lfloor \frac{k_j}{\alpha_j}\rfloor}} \notin \mathbb{R}$, the limit of $g_j$ belongs to a rank-2 cusp instead to a rank-1 cusp, giving us the contradiction. To obtain this, it is enough to pick $k_j$ such that $\frac{\alpha_j\lfloor \frac{k_j}{\alpha_j}\rfloor - k_j}{l_j\lfloor\frac{k_j}{\alpha_j}\rfloor} $ is uniformly bounded.

Let us pick $k_j$ in the following way:

\begin{itemize}
	\item $ \frac{l_j}{\alpha^2_j} \geq 1$:
	Take $k_j = 1$. Then $\vert l_j\lfloor \frac{k_j}{\alpha_j} \rfloor - \frac{l_j}{\alpha_j} \vert  \leq l_j$, and since $\lim l_j = \lim \frac{l_j}{\alpha_j} =0$, for these $j$'s we have $\lim  l_j\lfloor \frac{k_j}{\alpha_j} \rfloor = 0$.
	
	Now
	\begin{equation}
	 \left|\frac{\alpha_j\lfloor \frac{k_j}{\alpha_j}\rfloor - k_j}{l_j\lfloor\frac{k_j}{\alpha_j}\rfloor}\right| 
	 \leq \left|\frac{\alpha_j}{l_j\lfloor\frac{1}{\alpha_j}\rfloor}\right| 
	 \leq \max \left\lbrace\frac{\alpha^2_j}{l_j} , \frac{\alpha^2_j}{l_j(1-\alpha_j)} \right\rbrace
	\end{equation}
	is uniformly bounded
	
	\item $ \frac{l_j}{\alpha^2_j} < 1$:
	Take $k_j = \pm\lfloor \frac{\alpha^2_j}{l_j}\rfloor$, with the sign of $\alpha_j$. Then $\vert l_j\lfloor \frac{k_j}{\alpha_j} \rfloor \vert \leq  l_j \frac{\lfloor \frac{\alpha^2_j}{l_j}\rfloor}{\vert\alpha_j\vert} \leq \vert\alpha_j\vert$ converges to $0$ for such $j$'s.
	
	Now
	\begin{equation}
	\left|\frac{\alpha_j\lfloor \frac{k_j}{\alpha_j}\rfloor - k_j}{l_j\lfloor\frac{k_j}{\alpha_j}\rfloor}\right|
	\leq \left|\frac{\alpha_j}{l_j\lfloor\frac{k_j}{\alpha_j}\rfloor}\right| 
	\leq \frac{\vert\alpha_j\vert}{l_j\left( \frac{\frac{\alpha^2_j}{l_j}-1}{\vert\alpha_j\vert}-1\right) }
	\leq \frac{1}{1 - \frac{l_j}{\alpha^2_j} -\frac{l_j}{\vert\alpha_j\vert}}
	\end{equation}
	is uniformly bounded for a sufficiently large $j$.
\end{itemize}

The second condition follow easily if we pullback a fundamental polyhedron for $G$ to fundamental polyhedrons for $\{G_n\}$ as in \cite{JorgensenMarden} (convergence of metrics in compact subsets of the interior already follows from geometric convergence). To extend this to compact sets touching the boundary we use that $G$ is geometrically finite, so that a fundamental polyhedron can be pulled back. Observe that we look at convergence on compact sets when we compactify the non-toroidal ends. So we have proved:

\begin{lem}
If a sequence of convex co-compact metrics have a geometrically finite geometric limit, then the sequence is admissible.
\end{lem}

Now it's a matter of observing that the approach of \cite{MoroianuGuillarmouRochon} extends to our case. 

First let us fix some notations and statements from \cite{MoroianuGuillarmouRochon}. A hyperbolic cusp (with metric $h$) on a surface can be parametrized by
\begin{equation}
	\big( (0,\frac{1}{R})_v \times (\mathbb{R}/\frac{1}{2}\mathbb{Z})_w, h = \frac{dv^2}{v^2} + v^2dw^2  \big),
\end{equation}
 and capped to $[0,\frac{1}{R})_v \times (\mathbb{R}/\frac{1}{2}\mathbb{Z})_w$. With this, compactify any surface $(S,h)$ (where $h$ is a metric that has hyperbolic cusps) to a surface $\overline{S}$. Define as well $\dot{\mathcal{C}}^\infty(\overline{S})$ as the subspace of $\mathcal{C}^\infty(\overline{S})$ that vanishes to infinite order at $\partial\overline{S}$, and $\mathcal{C}^\infty_r(\overline{S})$ as the functions $f$ with $\partial_w f \in \dot{\mathcal{C}}^\infty(\overline{S})$ (where $w$ is the corresponding coordinate at each hyperbolic cusp).

\begin{prop}[Prop. 2.4 \cite{MoroianuGuillarmouRochon}]
Let $M$ be a geometrically finite hyperbolic $3$-manifold with rank-$1$ cusps. Let $(S,[h])$ be the conformal boundary and $h^\textnormal{hyp}$ be the complete hyperbolic metric in that conformal class. For each $\psi \in \mathcal{C}^\infty_r(\overline{S})$, consider the conformal representative $\hat h= e^{2\psi} h^{\textnormal{hyp}}$. There exists a smooth function $\hat\rho$ (in $\overline{M}_c$, the closure with corners) called geodesic boundary defining function, and a closed set $\mathcal{V}\subset\overline{M}_c$ with finite volume such that:
\begin{equation}
	\Big\vert \frac{d\hat\rho}{\hat\rho}\Big\vert_g =1 \textnormal{ in } \overline{M}_c \setminus\mathcal{V},  \hat\rho^2g\vert_S = \frac{\hat h}{4}.
\end{equation}
The set $\mathcal{V}$ does not intersect a compact subset of $S$.  The function $\hat\rho$ is defined uniquely near $\overline{S}$. The function $H(z) := \int_M \hat\rho^z \textnormal{dvol}_g$ admits a meromorphic extension from $\textnormal{Re}(z) > 2$ to a neighborhood of $z=0$ with a simple pole at $z=0$.
\end{prop}

They observe that, as is the convex co-compact case, the finite part of $H$ at $z=0$ does not depend on the values of $\hat\rho$ at $\mathcal{V}$, due to $\mathcal{V}$ being finite volume. More precisely:

\begin{equation}\label{eq:FP}
	\text{FP}_{z=0}H(z) = \Big( \text{FP}_{z=0}\int_{M\setminus\mathcal{V}} \hat\rho^z \textnormal{dvol} g\Big) + \text{vol}_g(\mathcal{V})
\end{equation}

Then with that setting they define the renormalized volume as:

\begin{defi} With the previous notation
\begin{eqnarray}
	\textnormal{Vol}_R (M,\hat h) :=  \text{FP}_{z=0}\int_{M} \hat\rho^z \textnormal{dvol}g\\
	\textnormal{Vol}_R (M) := \textnormal{Vol}_R (M, h^\textnormal{hyp})
\end{eqnarray}
\end{defi}

We alert the reader that we are using a slightly different renormalization than the one at \cite{MoroianuGuillarmouRochon}. In Proposition 2, they define $\hat\rho$ with $\hat\rho^2g\vert_S = \hat h$. This just changes the computations by a constant, but our choice has the property of making $\text{Vol}_R = V_C$ when the convex core has totally geodesic boundary, which is a more natural choice.

To see this fact, assume that the convex core of $g$ has totally geodesic boundary. Then $\hat\rho = e^{-r}$ is a geodesic boundary defining function for $h^\text{hyp}$ in our sense, where $r$ is the distance function to the convex core. Indeed, the level set of $\hat\rho$ are equidistant surfaces to the boundary, with metric $\cosh^2 (r) h^\text{hyp} $ (when we identify the surfaces with the normal geodesics). Take $\mathcal{V}$ as the convex core in \ref{eq:FP} to obtain:

\begin{align}\begin{aligned}\label{eq:FPcalculation}
	\text{Vol}_R(M) &= \text{vol}(\mathcal{V}) + \text{FP}_{z=0}\Big(\int_0^\infty\int_{S} (e^{-r})^z \cosh^2 (r) \textnormal{dvol}_{h^\text{hyp}}dr\Big)
	\\&= \text{vol}(\mathcal{V}) -2\pi\chi(\partial M) \text{FP}_{z=0}\Big( \frac{1}{z}\Big)
	\\&=\text{vol}(\mathcal{V})
	\end{aligned}
\end{align}

This renormalization is compatible with what is expressed in \cite{Schlenker13}, \cite{BridgemanCanary15}.

In order to prove existence of geodesic boundary defining functions, we can do the exact same parallel for $M$ also having only rank-$1$ cusps (as appears in \cite{MoroianuGuillarmouRochon}). The finite volume closed set $\mathcal{V}$ has the same properties, with the difference being that includes the rank-$2$ cusps on its interior.

Theorem \ref{theorem:main} proof follows word by word as in \cite{MoroianuGuillarmouRochon} when they show that for an admissible sequence $g_\epsilon$ with limit $g_0$:

\begin{equation}
	\lim_{\epsilon\rightarrow0}\text{FP}_{z=0}\Big(\int_{M} \hat\rho^z_\epsilon \text{dvol}g_\epsilon\Big) = \text{FP}_{z=0} \Big( \int_M \hat\rho^z_0 \textnormal{dvol}_{g_0}\Big)
\end{equation}

away and near the cusps. The main difference is that now $\mathcal{V}$ contains rank-$2$ cusps, but still has finite volume and lies away of compact subsets of the conformal boundary. Besides that, we will be rewriting  \cite{MoroianuGuillarmouRochon} to fill the details.

Before going to consequences of the continuity, let us dig a bit more between the relationship of $\text{Vol}_R$ and $V_C$, the volume of the convex core. In a similar fashion as when we show that they are equal when the convex core has totally geodesic boundary, we can show that $\text{Vol}_R(M,h^\text{Thu}) = V_C - \frac{1}{4}L(\beta_C)$, where $h^\text{Thu}$ is the Thurston metric at the conformal boundary and $L(\beta_C)$ is the length of the bending lamination $\beta_C$ of the convex core. Indeed, $e^{-r}$ is a geodesic boundary defining function for $h^\text{Thu}$, where $r$ is the distance to the convex core. Taking the convex core as $\mathcal{V}$, we can operate as in \ref{eq:FPcalculation}. For the geodesic faces of the convex core we have (\ref{eq:FPcalculation})  with $\text{FP}_{z=0}\Big(\frac{1}{4(z-2)} + \frac{1}{2z} + \frac{1}{z+2}\Big) = 0$. For the bending locus, the metric at each level set of $\rho$ is $\frac{e^{2r} + 1}{2}$ times the flat part of the Thurston metric, hence we have a contribution of $L(\beta_C)\text{FP}_{z=0}\Big(\frac{1}{2(z-2)} + \frac{1}{2z} \Big) = -\frac{1}{4}L(\beta_C)$.

Now we proceed as in \cite{BridgemanCanary15}.  Theorem 3.2 there (which follows from \cite{HerronMaMinda}) implies $\frac{h^\text{Thu}}{2} \leq h^\text{hyp} \leq h^\text{Thu}$. Using the same method as \cite{Schlenker13}, we can produce a 1-parameter family of metrics at infinity and use the variation formula to show that:

\begin{equation}
	\text{Vol}_R(M, \frac{h^\text{Thu}}{2}) \leq \text{Vol}_R(M, h^\text{hyp}) \leq \text{Vol}_R(M, h^\text{Thu}).
\end{equation}

 And because $\partial M$ is incompressible in $M$, $L(\beta_C) \leq \frac{\pi^3}{\sinh^{-1}1}\vert\chi(\partial M)\vert$ thanks to (\cite{BridgemanCanary05}).

Notice finally that for a constant $a$, $e^a\hat\rho$ is a geodesic boundary defining function for $e^{2a}\hat h$. Moreover, given (\cite{MoroianuGuillarmouRochon} Proposition 7.1), we have that:

\begin{equation}
	\text{Vol}_R(M,e^{2a}\hat h) = \text{Vol}_R(M,\hat h) + \frac{1}{2}\int_S a\text{dvol}_{h^\text{hyp}} = \text{Vol}_R(M,\hat h) - a\pi\chi(\partial M),
\end{equation}
which allows us to conclude:

\begin{theorem}\label{theorem:comparison}
Let $M$ be a geometrically finite hyperbolic manifold with incompressible boundary, then:
\begin{equation}
	V_C + \frac{\pi}{2}\ln(2)\chi(\partial M) + \frac{\pi^3}{\sinh^{-1}1}\chi(\partial M) \leq \textnormal{Vol}_R \leq V_C -  \frac{1}{4}L(\beta_C) \leq V_C
\end{equation}
\end{theorem}

Since $\text{Vol}_R = V_C$ implies $L(\beta_C)=0$, we also have: that $\text{Vol}_R = V_C$ if and only if the convex core has boundary totally geodesic. Notice that $\chi(\partial M)$ does not count tori components, and for the rest of them this number is negative.

\section{Consequences of the continuity of the renormalized volume under geometric limits}\label{sec:consequences}

\begin{theorem}
Let $M$ be a convex co-compact acylindrical hyperbolic  $3$-manifold. Then, among the deformation space of convex co-compact hyperbolic metrics, any sequence for which $V_\text{R}$ is going to its infimum converges to the geodesic class.
\end{theorem}

\begin{proof}
Assume the contrary. Then, since $g$ is the unique critical point for $V_\text{R}$, we have a sequence of convex co-compact metrics $g_j$  with $\lim V_\text{R} (g_j) = V$ that does not converge to a convex co-compact metric. Using \cite{BridgemanCanary15}, the volume of their convex core stays bounded, so we can assume that they converge geometrically to a geometrically finite hyperbolic manifold $M_0$ with metric $g_0$.  Let $\rho: \pi_1(M)  \rightarrow PSL(2,\mathbb{C})$ be the algebraic limit of this sequence. Since such a sequence is also admissible, we know by Theorem \ref{theorem:main} that $V_\text{R}(g_0)=V$.

From Section \ref{sec:sequences}, the limit is determined by drilling out a finite collection of unlinked (w.r.t. $\partial M$) closed curves and a conformal class at the boundary obtained by pinching some curves.

Now we are going to follow the description of convergence for Kleinian groups given by \cite{BBCL}. In summary, they describe the new parabolics as a sequence of shrinking or Dehn-twisting some curves at the conformal boundary (giving rank-1 and rank-2 cusps, respectively) in a way that this can be done to produce nearby hyperbolic metrics for the same topological space. 

By [\cite{BBCL} ,Theorem 4.1], we know that the immersion of a compact core of the algebraic limit $\rho$ into the geometric limit $M_0$ is an embedding outside a neighborhood of a collection of disjoint curves $q = \left\{ q_1, \ldots, q_n \right\}$ (called multicurve) at the boundary. The collar of an element $q_j$ of the multicurve wraps around a rank-2 parabolic of $M_0$. Let us examine one end at the time, so assume that we are dealing instead with a sequence of quasifuchsian manifolds where the conformal structure of the bottom end converges in Teichmuller space, and keep all the previous notation and assumptions.

As in \cite{BBCL}, we can identify the end invariants of a convex co-compact manifold with the multicurve (a union of disjoint curves)) given by a choice of a minimal pant decomposition of their conformal boundary, when they are considered with the corresponding hyperbolic metric. A \textit{generalized marking} of a multicurve is the multicurve itself together with a selection of transversal curves, at most one for each component of the multicurve (and tranverse to only one element of the multicurve). Name these end invariants markings by $\nu^+_j, \nu^-_j$, were the sign differentiates top from bottom.

The sequence of bottom end invariants for the sequence that converges to $\rho$ converges to the end invariants of the bottom peripheral subgroup of $M_0$. Hence we can consider the same end invariants for the terms of this sequence. In the case that rank-1 cusps are developed, they correspond to curves shrinking at the conformal boundary, so we include their corresponding curves along the sequence (observe that their length is going to 0 along the sequence).

In \cite{BBCL}, Section 1, is defined for an multicurve $\nu$, a curve $d$ and a marking $\mu$

\begin{equation}
	m(\nu,d,\mu) = \left\{ \sup_{d\subseteq Y}d_Y(\nu,\mu), \frac{1}{l_\nu(d)} \right\},
\end{equation}

where $d_Y(\nu,\mu)$ is the distance in the curve complex of two projections of $\nu$ and $\mu$ to a subsurface $Y$, and $l_\nu(d)$ is the length of the geodesic homotopic to $d$ in the hyperbolic metric defined by $\nu$.

It is also defined

\begin{equation}
	m^{na}(\nu,d,\mu) = \left\{ \sup_{\stackrel{d\subseteq Y}{ Y\neq\text{collar}(d)} }d_Y(\nu,\mu), \frac{1}{l_\nu(d)} \right\}.
\end{equation}

From [\cite{BBCL} Theorem 4.1], we can characterize combinatorially the parabolic appearing at the limit. Fix first a marking $\mu$. Then a curve $d$ will be a \textit{upward-pointing parabolic} if $\vert m(\nu^+_j,d,\mu)\vert - \vert m(\nu^-_j,d,\mu)\vert \rightarrow \infty$, \textit{downward-pointing parabolic} if $\vert m(\nu^-_j,d,\mu)\vert - \vert m(\nu^+_j,d,\mu)\vert \rightarrow \infty$, and a \textit{wrapped parabolic} if both $\vert m(\nu^-_j,d,\mu)\vert$, $\vert m(\nu^+_j,d,\mu)\vert$ go to $\infty$, while $m^{na}(\nu^+_j,d,\mu),m^{na}(\nu^-_j,d,\mu)$ stay bounded. Moreover, given (\cite{BBCL} Theorem 1.2), we know that a compact core embeds in the geometric limit if and only if there are no wrapped parabolics. Since for our sequence the bottom end invariants are the same, we see that we only could have upward-pointing parabolics. Hence a surface $S_0$ representing the algebraic limit embeds into the geometric limit, and by Proposition 1, the rank-2 parabolics lie unlinked and parallel between this surface and the top conformal boundary. Now we can isotope $S_0$ to the top conformal boundary so it crosses one rank-2 parabolic at the time and name the resulting surfaces after each step (which are not isotopic in the geometric limit) $S_1,\ldots, S_k$. The change between the groups that are represented by $S_0$ and $S_1$ are as described in \cite{KerckhoffThurston}, Section 3, where the algebraic sequence $\rho_j$ is changed by $\rho_j\circ D^1_j$, $D^1_j$ being the appropiate number of Dehn-twists around the curve that we are crossing (look also \cite{BBCL}, Lema 4.5). Hence the sequence to obtain $S_i$ is $\rho^i_j = \rho_j\circ D^1_j\circ\ldots\circ D^i_j$, with limit $\rho^i$. Observe that $S_k$ corresponds to the top peripheral subgroup, hence the top end invariants of $\rho^k_j$ converge to the top conformal boundary of $M_0$, so we can associate uniform top markings for them (say $\nu^{+,k}$). Then $\nu^{+,i}_j$ is the top end marking for the sequence $p^i_j$, where $D^{i+1}_j\circ\ldots\circ D^k_j(\nu^{+,i}_j) = \nu^{+,k}$. To end assumptions for our original sequence, assume that for each sequence $p^i_j$ the top and end markings bound projections. This is obtained thanks to \cite{BBCL} Theorem 1.1, after taking a subsequence. The definition of bounding projections is given in the paper, but for us it's only important that is equivalent (after taking subsequences) to the algebraic sequence being convergent.

While considering the deformation space for $M_0$, take a sufficiently small perturbation of the top conformal boundary so $\mu^{+,k} = \nu^{+,k}$ is still a marking for a sequence $\sigma^k_j$ converging to the small perturbation (consider the bottom conformal boundary to converge in Teichmuller space). Then $\mu^i_j$ is a marking for $\sigma^i_j =  \sigma_j\circ D^1_j\circ\ldots\circ D^i_j$. 

Take a subsequence if necessary so that $\sigma_j$ converges algebraically and geometrically (not necessarily to the same limit). These are the requirements of \cite{BBCL} (Theorem 1.2) for the combinatorial information to predict the end invariants of the algebraic limit. And since the combinatorial information is given by distance in the curve complex between a marking for the sequence and a fixed marking, we can easily conclude that $\sigma_j$ gives the same combinatorial information as $\rho_j$. In particular, the algebraic limit $\sigma$ embeds into its geometric limit and only has upward-pointing parabolics, if any. We have also same combinatorial information for $\sigma^i_j$ and $\rho^i_j$, so $\sigma^i_j$ bounds projections. After taking a subsequence, we can also assume that $\sigma^i_j$ is convergent to a representation $\sigma^i$. Now the process is pretty straighforward. The difference between consecutive $\sigma^i$ is a rank-2 parabolic in the same place as the rank-2 parabolic between the corresponding terms of $\rho^i$. At the final step, $\sigma^k_j$ has the top end invariant converging to the chosen small perturbation of the top conformal structure of $M_0$. Since the curves drilled out are the same, we obtained the chosen small perturbation of $M_0$.

Now, $g_0$ has to be a critical point for $V_\text{R}$, otherwise there is a nearby metric $g'_0$ with $V_\text{R}(g'_0) < V$ which is the geometric limit of a sequence of convex co-compact metrics on $M$. Since $V_\text{R}$ is continuous under geometric limit by Theorem \ref{theorem:main}, we will have that $V_\text{R}(g'_0) \geq V $, which is a contradiction.  Because \cite{MoroianuGuillarmouRochon} (Theorem 2), $g_0$ is totally geodesic, then $V_\text{R} (g_0)$ is equal to $V_C(g_0)$, where $V_C$ is the volume of the convex core. But since we have $V_R(g)= V_C(g) $, then $V_C(g) \geq V_C(g_0)$. Since both convex cores are totally geodesic, $V_C$ is half the volume of the doubled manifolds. But such inequality is now impossible since cusps strictly increase the volume (Theorem 6.5.6 \cite{Thurstonnotes}).
\end{proof}

Notice that we can isolate the use of $M$ being acylindrical in the following properties:

\begin{enumerate}
	\item Boundary incompressible with any essential simple closed curve indivisible in the interior.
	\item The existence of a geodesic class.
	\item For any sequence there is a subsequence that converges algebraically.
\end{enumerate}

Then we can re-do the argument to show continuity for $M$ quasifuchsian if we assume that the sequence converges algebraically. This is in general not true for sequences with geometrically finite limit. For example, we can shrink the length of a given separating curve at both ends (w.r.t. the hyperbolic metric in the conformal class at infinity). This makes the geometric limits distinct if we put base points at different sides of the convex core with respect to this curve, so there is not algebraically convergent subsequence. However, we can assure algebraic convergence of a subsequence if we for example move in the same Bers slice. Sadly, this is not enough to carry the analogous argument to show minimality along the Bers slice. Indeed, when we reach the last paragraph of the preceding proof, the limit metric will be only be critical while variating the top end, so we can only assure that the boundary of the convex core corresponding to this end will be totally geodesic. This in principle is not enough to show that its renormalized volume will be positive (the objective being to prove that the minimum of $V_\text{R}$ is 0 and corresponds only to the fuchsian metric on the Bers slice).

\bibliographystyle{amsalpha}
\bibliography{mybib}

\end{document}